\documentclass[a4paper,12pt]{amsart}
\usepackage[frenchstyle,largesmallcaps,fulloldstylenums,nott,easyscsl]{kpfonts}

\makeatletter
\def\@settitle{\baselineskip14\p@\relax
    \flushleft{\Large\bfseries \@title}
}
\def\@setauthors{%
  \begingroup
  \def\thanks{\protect\thanks@warning}%
  \trivlist
  \centering\footnotesize \@topsep30\p@\relax
  \advance\@topsep by -\baselineskip
  \item\relax
  \author@andify\authors
  \def\\{\protect\linebreak}%
  \flushleft{\itshape by \authors\\ with an appendix by Nir \textsc{Lazarovich}}%
  \ifx\@empty\contribs
  \else
    ,\penalty-3 \space \@setcontribs
    \@closetoccontribs
  \fi
  \endtrivlist
  \endgroup
}

\def\maketitle{\par
  \@topnum\z@ % this prevents figures from falling at the top of page 1
  \@setcopyright
  \thispagestyle{firstpage}% this sets first page specifications
  \ifx\@empty\shortauthors \let\shortauthors\shorttitle
  \else \andify\shortauthors
  \fi
  \@maketitle@hook
  \begingroup
  \@maketitle
  \author@andify\authors
  \toks@\@xp{\authors}\@temptokena\@xp{\shorttitle}%
  \toks4{\def\\{ \ignorespaces}}% defend against questionable usage
  \edef\@tempa{%
    \@nx\markboth{\the\toks4
      \the\toks@}{\the\@temptokena}}%
  \@tempa
  \endgroup
  \c@footnote\z@
  \@cleartopmattertags
}

\def\ps@headings{\ps@empty
  \def\@evenhead{%
    \setTrue{runhead}%
    \normalfont\scriptsize
    \llap{\thepage\hskip1em}
    \def\thanks{\protect\thanks@warning}%
    \leftmark{}{}\hfil}%
  \def\@oddhead{%
    \setTrue{runhead}%
    \normalfont\scriptsize \hfil
    \def\thanks{\protect\thanks@warning}%
    \rightmark{}{}\rlap{\hskip1em\thepage}}%
  \let\@mkboth\markboth
}
\ifx\ParFact\undefined\def\ParFact{0.4}\fi
\def\@setparskip{\parskip=\ParFact\baselineskip \advance\parskip by 0pt plus 2pt}
\@setparskip
\let\oldtableofcontents=\tableofcontents
\def\tableofcontents{\parskip=0pt \oldtableofcontents \@setparskip}

\def\section{\@startsection{section}{1}{\z@}%
    {-21dd plus-4pt minus-4pt}{10.5dd plus 4pt
     minus4pt}{\normalsize\bfseries\boldmath}}
\def\specialsection{\@startsection{section}{1}{\z@}%
    {-21dd plus-4pt minus-4pt}{10.5dd plus 4pt
     minus4pt}{\normalsize\bfseries\boldmath}}
\makeatother

\def\presp#1#2%
  {\mathop{}%
   \mathopen{\vphantom{#2}}^{#1}%
   \kern-\scriptspace%
   #2}

\usepackage[flushmargin]{footmisc} 
\usepackage{varwidth}
\usepackage{enumitem}
\setlist[enumerate]{labelindent=--.5in,leftmargin=0pt}
\setlist[itemize]{labelindent=--.5in,leftmargin=0pt}
\usepackage{nicefrac}
\usepackage{euscript}
\usepackage{color}
\definecolor{green}{RGB}{0,127,0}
\definecolor{red}{RGB}{191,0,0}
\usepackage[colorlinks=true]{hyperref}

\theoremstyle{plain}

\newtheorem{theorem}{Theorem}[section]
\newtheorem{lemma}[theorem]{Lemma}
\newtheorem{corollary}[theorem]{Corollary}
\newtheorem{proposition}[theorem]{Proposition}

\numberwithin{figure}{section}
\theoremstyle{definition}
\newtheorem{example}[theorem]{Example}
\newtheorem{remark}[theorem]{Remark}
\newtheorem{question}[theorem]{Question}
\numberwithin{equation}{section}

\DeclareMathOperator{\stab}{Stab}
\DeclareMathOperator{\aut}{Aut}
\DeclareMathOperator{\aaut}{AAut}
\DeclareMathOperator{\homeo}{Homeo}
\DeclareMathOperator{\supp}{supp}
\DeclareMathOperator{\conv}{Conv}
\def\bound{six}
\def\boundd{nine}

\newcommand{\B}[1]{{\mathbf #1}}
\newcommand{\C}[1]{{\mathcal #1}}

\newcommand{\G}[1]{{\llbracket #1\rrbracket}}
\newcommand{\textSC}[1]{\textsc{\MakeLowercase{#1}}}

\newcommand{\Th}{G} %Higman-Thompsons groups G series
\newcommand{\Fh}{F} %Higman-Thompsons groups F series
\newcommand{\Nr}{N} %Neretin group
\newcommand{\Zq}{\B Z[\nicefrac1q]}
\newcommand{\lhdneq}{\mathrel{\ooalign{$\lneq$\cr\raise.22ex\hbox{$\lhd$}\cr}}}
\newcommand{\id}{\mathrm{id}}

\begin{document}

\title{Uniform simplicity of groups with proximal action}
\author{{\'S}wiatos{\l}aw R.~\textsc{Gal}}
\address{S.R.~Gal --- \normalfont Instytut Matematyczny Uniwersytetu Wroc{\l}awskiego, pl. Grunwaldzki \nicefrac24, 50-384 Wroc{\l}aw, Poland
\& Weizmann Institute of Science, Rehovot 76100, Israel}
\email{sgal@math.uni.wroc.pl}
\thanks{The research leading to these results has received funding from the European Research Council under the European Unions Seventh Framework Programme (FP7/2007- 2013)/ERC Grant Agreement No. 291111. The first author partially supported by Polish National Science Center \textSC{(NCN)} grant \textSC{2012/06/A/ST1/00259}
and the European Research Council grant No. 306706.}
\author{Jakub \textsc{Gismatullin}}
\address{J.~Gismatullin --- \normalfont Instytut Matematyczny Uniwersytetu Wroc{\l}awskiego, pl. Grunwaldzki \nicefrac24, 50-384 Wroc{\l}aw, Poland
\& Instytut Matematyczny Polskiej Akademii Nauk, ul. {\'S}niadeckich 8,
00-656 Warszawa, Poland}
\email{gismat@math.uni.wroc.pl}%, www.math.uni.wroc.pl/\~{}gismat}
\thanks{\noindent The second author is partially supported by the \textSC{NCN} grants
\textSC{2014/13/D/ST1/03491, 2012/07/B/ST1/03513}.}
\address{N.~Lazarovich --- \normalfont Departement Mathematik, Eidgen\"ossische Technische Hochschule Z\"urich, R\"amistrasse 101, 8092 Z\"urich, Shweiz}
\email{nir.lazarovich@math.ethz.ch}



\begin{abstract}
We prove that groups acting boundedly and order-primitively on linear orders or
acting extremely proximally on a Cantor set (the class including various Higman-Thomson groups;
Neretin groups of almost automorphisms of regular trees, also called groups of
spheromorphisms; the groups of quasi-isometries and almost-isometries of regular trees) are uniformly simple. Explicit bounds are provided.
\end{abstract}

\maketitle

\tableofcontents

\section{Introduction}

Let $\Gamma$ be a group.
It is called \textbf{$N$-uniformly simple} if for every nontrivial $f\in\Gamma$
and nontrivial conjugacy class $C\subset\Gamma$ the element $f$ is the product
of at most $N$ elements from $C^{\pm1}$.
%, i.e. \[G = \left(g^G \cup {g^{-1}}^G\right)^{\leq N}.\]
A group is \textbf{uniformly simple} if it is $N$-uniformly simple for some
natural number $N$.
Uniformly simple groups are called sometimes, by other authors, \textit{groups
with finite covering number} or \textit{boundedly simple} groups (see e.g.
\cite{MR1710745, MR1919378, MR3085032}).  
We call $\Gamma$ \textbf{boundedly simple}, if $N$ is allowed to depend on $C$.
%This notion will appear only in Proposition \ref{prop:norm}.
The purpose of this paper is to prove results on uniform simplicity, in particular Theorems \ref{thm:doubly-transitive}, \ref{thm:Neretin}, and \ref{thm:tree} below, for
a number of naturally occurring infinite permutation groups.

Every uniformly simple group is simple. It is known that many groups with
geometric or combinatorial origin are simple. In this paper we prove that,
in fact, many of them are uniformly simple. 

Below are our main results.  

Let $(I,\leq)$ be a linearly ordered set.  Let
$\mathrm{Aut}(I,\leq)$ denote the group of order preserving bijections of $I$.
We say that $g\in\mathrm{Aut}(I,\leq)$ is \textbf{boundedly supported} if there are
$a, b\in I$ such that $g(x)\neq x$ only if $a<x<b$.  The subgroup of boundedly
supported elements of $\mathrm{Aut}(I,\leq)$ will be denoted by $\mathrm{B}(I,\leq)$.

\begin{theorem}[Theorem \ref{thm:order} below]\label{thm:doubly-transitive}
Assume that $\Gamma<\mathrm{B}(I,\leq)$ is proximal on a linearly ordered set $(I,\leq)$
(i.e. for every $a<b$ and $c<d$ from $I$ there exists $g\in\Gamma$ such that $g(a)<c<d<g(b)$).
Then its commutator group $\Gamma'$ is
\bound-uniformly simple and the commutator width of this group is at most two.
\end{theorem}

For the definition of a \textbf{commutator width of a group} see the beginning of Section~\ref{sec:burago}.
Observe that every doubly-transitive (i.e. transitive on ordered pairs) action is proximal.

This theorem immediately applies e.g. to $\mathrm{B}\left(\B Q,\leq\right)$ and the class of Higman-Thomson groups
$F_{q,r}$, for $q>r\geq 1$, where the latter are defined as follows.  We fix natural numbers
$q>r\geq 1$.  The \textbf{Higman-Thompson group} $\Fh_{q,r}$ is defined as
piecewise affine, order preserving transformations of $\left((0,r)\cap\Zq,\leq\right)$
whose breaking points (i.e. singularities) belong to $\Zq$
and the slopes are $q^k$, for $k\in\B Z$ (see \cite[Proposition 4.4]{MR885095}).
The \textbf{Thompson group $F$} is the group $\Fh_{2,1}$ in
the above series. Moreover, $\Fh_{q,r}$ is independent of $r$ (up to
isomorphism) \cite[4.1]{MR885095}. The Higman-Thompson groups satisfy the assumptions
of Theorem \ref{thm:doubly-transitive} due to Lemmata \ref{lem:ht},
\ref{lem:id} from Section \ref{sec:order}.

Our result implies that $\Gamma=B(\mathbf{Q},\leq)$ is six-uniformly simple.
Whereas Droste and Shortt proved in \cite[Theorem 1.1(c)]{MR1089954} that
$B(\B Q,\leq)$ is two-uniformly simple.  In fact, they proved that if
$\Gamma<\mathrm{B}(I,\leq)$ is proximal (they use the term `feebly
2-transitive' for proximal action) and additionally closed under
$\omega$-patching of conjugate elements, then $\Gamma$ is two-uniformly simple.
Thus, our Theorem \ref{thm:doubly-transitive} covers larger class of examples
than that from \cite{MR1089954} (as we assume only proximality), but with
slightly worse bound for uniform simplicity.

The uniform simplicity of the Thompson group $F=F_{2,1}$ was proven implicitly by Bardakov, Tolstykh,
and Vershinin \cite[Corollary 2.3]{MR2959420} and Burago and Ivanov
\cite{MR2478468}. Although their proofs generalise to the
general result given above, we write it down for several reasons.  Namely, in
above cited papers some special properties of the linear structure of the real line is
used, while the result is true for a general class of proximal actions on
linearly ordered sets. The Droste and Shortt argument uses $\omega$-patching,
which is not suitable for our case. Furthermore, although in
the examples mentioned above the action is doubly-transitive, the right
assumption is proximality, which is strictly weaker than double-transitivity.
In Theorem \ref{ex:proximal-not-2-transitive} we construct a bounded and proximal
transitive action which is not doubly-transitive.  This is discussed
in details in Section \ref{sec:para}.
The second reason for proving Theorem \ref{thm:doubly-transitive} is that a
topological analogue of proximality, namely extremal proximality
(see the beginning of Section \ref{sec:Cantor}),
plays a crucial role in the proofs of the subsequent results.
Extremal proximality was defined by S.~Glasner in \cite[p.  96]{MR0474243} and
\cite[p. 328]{MR0336723} for a general minimal action of a group on a compact
Hausdorff space.

In Section \ref{sec:Cantor} we go away from order preserving actions, and 
consider groups acting on a Cantor set, and also groups almost acting on trees.
The following theorem is Corollary \ref{cor:ner}(2).

\begin{theorem} \label{thm:Neretin}
The commutator subgroup $\Nr'_q$ of the Neretin group $\Nr_q$ of spheromorphisms and the commutator subgroup
$\Th'_{q,r}$ of the Higman-Thomson group $\Th_{q,r}$ are \boundd-uniformely
simple.  The commutator width of each of those groups is at most three.
\end{theorem}

The group $\Nr_q$ was introduced by Neretin in \cite[4.5, 3.4]{MR1209033} as
the group of sphereomorphisms (also called almost automorphisms) of a
$(q+1)$-regular tree $T_q$.  We will recall the construction in Section 
\ref{sec:Neretin}.

The \textbf{Higman-Thompson group} $\Th_{q,r}$ is defined as the group of
automorphisms of the J{\'o}nsson-Tarski algebra $V_{q,r}$
\cite[4\textSC{A}.]{MR885095}. It can also be described as a certain group of
homeomorphisms of a Cantor set \cite[p. 57]{MR885095}. Moreover, one can view
$\Th_{q,r}$ as a subgroup of $\Th_{q,2}$ and the latter as a group acting sphereomorphically
on the $(q+1)$-regular tree \cite[Section 2.2]{MR1703086}, that is, they are subgroups
of $\Nr_q$. If $q$ is even, then $\Th'_{q,r}=\Th_{q,r}$. For odd $q$,
$[\Th_{q,r}:\Th'_{q,r}] = 2$ \cite[Theorem 4.16]{MR885095},  \cite[Theorem
5.1]{1502.00991}.  The group $\Th_{2,1}$ is also known as the \textbf{Thompson
group $V$}. It is known that $\Fh_{q,r} < \Th_{q,r}$.

Given a group $\Gamma$ acting on a tree $T$, in the beginning of Section
\ref{sec:Cantor} we will define, following Neretin, the group $\G\Gamma$ of
partial actions on the boundary of $T$.  Theorem \ref{thm:Neretin} is a
corollary of a more  general theorem about uniform simplicity of partial
actions.

\begin{theorem} \label{thm:tree}
Assume that a group $\Gamma$ acts on a leafless tree $T$, whose boundary is
a Cantor set, such that $\Gamma$ does not fix any proper subtree (e.g.
$\Gamma\backslash T$ is finite) nor a point in the boundary of\/ $T$.  Then the
commutator subgroup ${\G{\Gamma}}'$ of\/ $\G{\Gamma}$ is \boundd-uniformly
simple.
\end{theorem}

This is an immediate corollary of Theorems \ref{thm:C} and \ref{thm:parr}. The latter
theorem concerns with several
characterisations of extremal proximality of the group action on the boundary of a tree.

Section \ref{sec:app} is devoted to the proof that the groups of
quasi-isometries and almost-isometries of regular trees are five-uniformly
simple.

The uniform simplicity of homeomorphism groups of certain spaces has been
considered since the beginning of 1960s e.g. by Anderson \cite{MR0139684}.
He proved that the group of all homeomorphisms of ${\mathbf R}^n$ with compact
support and the group of all homeomorphisms of a Cantor set are
two-uniformly simple (and has commutator width one). His arguments uses an
infinite iteration arbitrary close to every point, which is not suitable for
the study of spheromorphisms group and the Higman-Thomson groups. 

$N$-uniform simplicity is a first order logic property (for a fixed natural
number $N$). That is, can be written as a formula in a first order logic. Therefore,
$N$-uniform simplicity is preserved under  elementary equivalence: if $G$ is $N$-uniform simple, then all other groups, elementary equivalent with $G$ are also $N$-uniformly simple. In particular, all ultraproducts of Neretin groups, and  Higman-Thompson groups mentioned above, are \boundd-uniformly
simple.

Another feature of uniformly simple groups comes from \cite[Theorem 4.3]{MR3085032}, where the second author proves the following classification fact about actions of uniformly simple groups (called boundedly simple in \cite{MR3085032}) on trees: if a uniformly simple group acts faithfully on a tree $T$ without invariant proper subtree or invariant end, then essentially $T$ is a bi-regular tree (see Section \ref{sec:Neretin} for definitions).

In Section
\ref{sec:para} we discuss \textit{the primitivity} of actions on linearly ordered sets (i.e. lack of proper convex congruences). In fact, we prove that primitivity and proximality are equivalent notions for bounded actions (Theorem \ref{thm:compa}). Our primitivity appears in the literature
as \textit{o-primitivity}, see e.g. \cite[Section 7]{MR1486196}.

Calegari \cite{MR2366352} proved that subgroups of $\textSC{PL}^+$
homeomorphism of the interval, in particular the Thompson group $F$, have
trivial stable commutator length.  Essentially by \cite[Lemma 2.4]{MR2509711}
by Burago, Polterovich, and Ivanov (that we will explain for the completeness of the
presentation) we reprove in Lemma \ref{lem:bip} the commutator width of the commutator 
subgroup (and other groups covered by Theorems \ref{thm:order} and 
\ref{thm:doubly-transitive}) of the Thompson group $F$.

Let us discuss examples and nonexamples of bonded and uniform simplicity.
It is known that a simple Chevalley group (that is, the group of points over an
arbitrary infinite field $K$ of a quasi-simple quasi-split connected reductive
group) is uniformly simple \cite{MR1710745, MR1919378}. In fact, there exists
a constant $d$ (which is conjecturally $4$, at least in the algebraically
closed case \cite{MR1327877}) such that, any such Chevalley group $G$ is $d
\cdot r(G)$-uniformly simple, where $r(G)$ is the relative rank of $G$
\cite{MR1710745}.

Full automorphism groups of right-angled buildings are simple, but never
boundedly simple, because of existence of nontrivial quasi-morphisms
\cite{MR3164745}, \cite[Theorem 1.1]{MR2585575} (except if the building is a
bi-regular tree \cite[Theorem 3.2]{MR3085032}).

Compact groups are never uniformly simple. More generally, a topological group
$\Gamma$ is called a \textit{\textsc{sin} group} if every neighborhood of the identity
$e\in\Gamma$ contains a neighborhood of $e$ which is invariant under all inner
automorphisms. Every compact group is \textsc{sin} (as if $V$ is such a neighborhood,
then $\bigcap_{\gamma\in\Gamma}\gamma^{-1}V\gamma$ has nonempty interior).
Clearly every infinite Hausdorff \textsc{sin}-group is not uniformly simple.  Moreover,
many compact linear groups (e.g. $SO(3,\B R)$) are boundedly simple, because of
the presence of the dimension with good properties.  (See also the discussion at
the end of this section.)

Let us conclude the introduction by relating simplicity and the notion of central norms
on groups.
Let $\Gamma$ be a group. A function $\|\cdot\|\colon\Gamma\to\B R_{\geq 0}$ is called a \textbf{seminorm} if
\begin{itemize}
\item $\|g\|=\|g^{-1}\|$ for all $g\in\Gamma$, and
\item $\|gh\|\leq\|g\|+\|h\|$ for all $g,h\in\Gamma$.
\end{itemize}
A seminorm is called
\begin{itemize}
\item \textbf{trivial} if $\|g\|=0$ for all $g\in\Gamma$,
\item \textbf{central} if $\|gh\|=\|hg\|$,
\item a \textbf{norm} if $\|g\|>0$ for $1\neq g\in\Gamma$,
\item \textbf{discrete} if $\inf_{1\neq g\in\Gamma}\|g\|>0$, and
\item \textbf{bounded} if $\sup_{g\in\Gamma}\|g\|<\infty$.
\end{itemize}
A seminorm is discrete if and only if the 
topology it induces is discrete. A discrete seminorm is a norm. Every central seminorm $\|\cdot\|$ is conjugacy invariant: $\|ghg^{-1}\|=\|h\|$.

A typical example of a nontrivial central and discrete norm is a \textbf{word norm $\|\cdot\|_S$} attached to a subset $S\subseteq \Gamma$ (cf. \cite[2.1]{MR2819193}): 
$$
\| f \|_S = \min\left\{ k\in \B N : f = g_1\cdot\ldots\cdot g_k,\text{ each } g_i \text{ is conjugate with an element from } S\cup S^{-1}\right\}.
$$
For a nontrivial central norm $\|\cdot\|$ we define
an invariant $\Delta(\|\cdot\|)=\frac{\sup_{g\in\Gamma}\|g\|}{\inf_{1\neq g\in\Gamma}\|g\|}$.
Of course, if $\|\cdot\|$ is either nondiscrete or unbounded then $\Delta(\|\cdot\|)=\infty$.
We define $\Delta(\Gamma)$ to be the supremum of $\Delta(\|\cdot\|)$ for all nontrivial
central norms on~$\Gamma$.

\begin{proposition}\label{prop:norm}
Let $\Gamma$ be a group.  Then
\begin{enumerate}
\item $\Gamma$ is simple if and only if any nontrivial central seminorm on $\Gamma$ is a norm;
\item $\Gamma$ is boundedly simple if and only if every central seminorm on $\Gamma$ is a bounded norm;
\item If\/ $\Gamma$ is uniformly simple, then every central seminorm on $\Gamma$ is a bounded and discrete norm;
\item $\Gamma$ is $N$-uniformly simple if and only if $\Delta(\Gamma)\leq N$.
\end{enumerate}
\end{proposition}

\begin{proof}
$(1)$ It is obvious that the kernel of a central seminorm is closed under multiplication and conjugacy invariant.  Thus, it is a normal subgroup. On the other hand, if\/ $\Gamma_0\lhdneq\Gamma$ is a proper normal subgroup of\/ $\Gamma$ then
$$
\|g\|=
\begin{cases}
0&\text{if $g\in\Gamma_0$,}\\
1&\text{if $g\not\in\Gamma_0$}\\
\end{cases}
$$
is a nontrivial central seminorm.  It is a norm only if\/ $\Gamma_0=\{1\}$.

$(2)$ Suppose that $\|\cdot\|$ is a central seminorm and assume that $\Gamma$
is boundedly simple. Choose $1\neq g\in\Gamma$.  There exists $N=N(g)$ such
that every element $f$ is a product of at most $N$ conjugates of $g$ and
$g^{-1}$.  Thus, by the triangle inequality and conjugacy invariance,
$\|f\|\leq N\|g\|$.  The number $N\|g\|$ is finite and independent on $f$. For
the converse take $1\neq g\in\Gamma$ and consider the word norm $\|\cdot\|_S$ attached to $S=\presp{\Gamma}g\cup \presp{\Gamma}{g^{-1}}=
\left\{x\in\Gamma : x\text{ is conjugated to }g\text{ or }g^{-1}\right\}$. It is obvious that $\|\cdot\|_S$ is a central seminorm on $\Gamma$. Thus
$\Gamma$ is boundedly simple, as $\|\cdot\|_S$ is bounded.

$(3\ \&\ 4)$ Suppose $\Gamma$ is $N$-uniformly simple, i.e. $N$ is independent
on $g\in\Gamma$ and take nontrivial central norm $\|\cdot\|$. By the triangle
inequality we conclude that $\Delta(\|\cdot\|)\leq N$, which proves the
necessity of the condition in $(4)$. For the converse, take $1\neq g\in\Gamma$,
and consider the word norm $\|\cdot\|_{g}$ above. We have that $\Gamma$ is
$\Delta(\|\cdot\|_g)\leq\Delta(\Gamma)$-uniformly simple.  This completes the
proof of (4), which implies (3).
\end{proof}

In particular, the infinite alternating group $A_\infty=\bigcup_{n\geq5}A_n$ is simple
but not boundedly simple.  To see the later, observe that the cardinality of the
support is an unbounded central norm.  This norm is maximal up to scaling.
Indeed, essentially by \cite[Lemma 2.5]{MR2178069} every element $\sigma\in A_\infty$
is a product of at most $\left\lfloor\frac{8\#\supp(\sigma)}{\#\supp(\pi)}\right\rfloor+2\leq\frac{10\#\supp(\sigma)}{\#\supp(\pi)}$
conjugates of $\pi$.

Also, it is easy to see that $\mathrm{SO}(3)$ is boundedly simple, but not uniformly simple.
The angle of rotation is a full invariant of an element of that group, and this function
is a central norm.  Clearly it is not discrete.  As before, one can observe that if
$R$ is a rotation by an angle $\theta$, then every other rotation can be obtained
by at most $\left\lceil\frac{\pi}{\theta}\right\rceil$ conjugates of $R$.

Every universal sofic group \cite[Section 2]{MR2178069} is boundedly simple,
but not uniformly simple. Namely, let $(S_n,\|\cdot\|_H)_{n\in\B N}$ be the
full symmetric group on $n$ letters with the normalized Hamming norm:
$\|\sigma\|_H=\frac{1}{n}|\supp(\sigma)|$, for $\sigma\in S_n$.  Take
\[({\mathcal G},\|\cdot\| ) = {\prod\nolimits_{\mathcal U}}^{met}(S_n,\|\cdot\|_H)\] the
metric ultraproduct of $(S_n,\|\cdot\|_H)_{n\in\B N}$ over a nonprincipal
ultrafilter $\mathcal U$. Then, the proof of
\cite[Proposition 2.3(5)]{MR2178069} shows that $\mathcal G$ is boundedly
simple. Furthermore, $\mathcal G$ is not uniformly simple, as $\|\cdot\|$ is a
non-discrete central norm on $\mathcal G$ (see Proposition
\ref{prop:norm}$(4)$).  As before, this norm is maximal up to scaling
due to \cite[Lemma 2.5]{MR2178069}.

\section{Burago--Ivanov--Polterovich method}\label{sec:burago}

The symbol $\Gamma$ will always denote a group. For $a,b \in \Gamma$ we use the following notation:
$\presp{g}h:=ghg^{-1}$ and $[g,h]:=\presp gh.h^{-1}=g.\presp hg^{-1}$.
By $\presp{\Gamma}g$ we mean the conjugacy class of $g\in\Gamma$.

Let $C$ be a nontrivial conjugacy class in $\Gamma$.  By
\textbf{$C$-commutator} we mean an element of $[\Gamma,C]=\{[g,h] : g\in\Gamma,h\in C\}$.
If $h\in C$ we will use the name $h$-commutator as a synonym of $C$-commutator, for short.
Of course $[\Gamma,C]=C.C^{-1}$, thus the set of $C$-commutators is closed
under inverses and conjugation.

The \textbf{commutator length} of an element $g \in [\Gamma,\Gamma]$ is the
minimal number of commutators sufficient to express $g$ as their product.  The
\textbf{commutator width} of\/ $\Gamma$ is the maximum of the
commutator lengths of elements of its derived subgroup $[\Gamma,\Gamma]=\Gamma'$.

We say that $f$ and $g\in\Gamma$ \textbf{commute up to conjugation} if
there exist $h\in\Gamma$ such that $f$ and $\presp{h}g$ commute.

%Due to the following lemma we do not have to specify in the subsequent statements 
%whether we deal with conjugacy classes in $\Gamma$ or in $\Gamma'$.

%\begin{lemma}\label{lem:class}
%Assume that $h\in\Gamma$ commutes up to conjugation with every element of\/ $\Gamma$.
%Then $\presp{\Gamma}h$ and $\presp{\Gamma'}h$ are equal.
%\end{lemma}

%\begin{proof}
%Let $f\in\Gamma$ and let $g\in\Gamma$ be such that $h$ and $\presp gf^{-1}$ commute.  Then
%$$
%\presp{[f,g]}h=\presp{f}{\left(\presp{\presp{g}f^{-1}}h\right)}=\presp{f}h.
%$$
%\end{proof}
\begin{lemma}\label{lem:commutators}
Assume that $\alpha$ and $\presp h\beta$ commute.
Then $[\alpha,\beta]$ is a product of two $h$-commutators.
More precisely $[\alpha,\beta]$ can be written as a product
of two conjugates of $h$ and two conjugates of $h^{-1}$
by elements from the group generated by $\alpha$ and $\beta$.
\end{lemma}

\begin{proof}
We have
$
[\alpha,[\beta,h]]=\left[\presp{\alpha}\beta,\presp{\alpha}h\right]\left[\beta^{-1},\presp{\beta}h\right]
$.
Also,
$
[\alpha,[\beta,h]]=\left[\alpha,\beta\ \presp{h}\beta^{-1}\right]=[\alpha,\beta]
$,
since $\alpha^{-1}$ commute with $\presp{h}\beta^{-1}$.
%By Lemma \ref{lem:class} they are $h$-commutators.
\end{proof}

Following Burago, Ivanov and Polterovich \cite[Sec. 2.1]{MR2509711} assume that
$H<\Gamma$ is a subgroup, $h\in\Gamma$ and $k\in\B N\cup\{\infty\}$. We say
that an element $h$ \textbf{$k$-displaces} $H$ if
$$
\left[f,\presp{h^j}g\right]=e\text{ for all }f,g\in H,\text{ and }j=1,\ldots,k
$$
(hence also $\left[\presp{h^i}f,\presp{h^j}g\right]=e$ for $1\leq |i-j|\leq k$).

We will say that $h$ \textbf{displaces} $H$ if it 1-displaces $H$.  We say that
$H<\Gamma$ is \textbf{$k$-displaceable} in $\Gamma$ if there exists
$h\in\Gamma$ such that $h$ $k$-displaces $H$ (this property is called
\textit{strongly $k$-displaceable} in \cite[Sec. 2.1]{MR2509711}). In
particular, elements of a displaceable subgroup commute up to conjugation.

\begin{lemma}\label{lem:bip}
\cite[Lemma 2.5]{MR2509711}
Assume that $h\in\Gamma$ $k$-displaces $H<\Gamma$. Let $f\in H'$ be
a product of at most $k$ commutators ($k\geq 2$).  Then there exists
$\alpha$, $\beta$, and $\gamma\in\Gamma$ such that $f=[\alpha,\beta][\gamma,h]$.
\end{lemma}
%\begin{proof} 
%Let $f=\prod_{i=0}^{k-1}\gamma_i$, where $\gamma_i=[\alpha_i,\beta_i]$, $\alpha_i, \beta_i\in H$.  We check that
%\begin{align*}
%\left[\prod_{i=0}^{k-1}\presp{h^i}{\left(\prod_{j=i}^{k-1}\gamma_j\right)},\,h\right]
%&=\prod_{i=0}^{k-1}\presp{h^{i+1}}{\left(\left(\prod_{j=i+1}^{k-1}\gamma_j\right)\left(\prod_{j=i}^{k-1}\gamma_j\right)^{-1}\right)}\cdot \left(\prod_{j=0}^{k-1}\gamma_j\right)\\
%& =\left(\prod_{i=0}^{k-1}\presp{h^{i+1}}{\gamma_i}^{-1}\right)\left(\prod_{j=0}^{k-1}\gamma_j\right)=\left(\prod_{i=0}^{k-1}\presp{h^{i+1}}{\gamma_i}^{-1}\right)f.\\
%\end{align*}
%Define $\alpha=\prod_{i=0}^{k-1}\presp{h^{i+1}}{\alpha_i}$,
%$\beta=\prod_{i=0}^{k-1}\presp{h^{i+1}}{\beta_i}$, and
%$\gamma=\prod_{i=0}^{k-1}\presp{h^i}{\left(\prod_{j=i}^{k-1}\gamma_j\right)}$.
%Then $[\gamma,h]f^{-1}=\prod_{i=0}^{k-1}\presp{h^{i+1}}{\gamma_i}^{-1}=[\alpha,\beta]^{-1}$ (since $\presp{h^{i+1}}x$ and $\presp{h^{j+1}}y$ commute for $0 \leq i\neq j \leq k-1$, $x,y\in H$).  Thus the claim.
%\end{proof}

Burago, Polterovich, and Ivanov \cite[Theorem 2.2(i)]{MR2509711} proved that if for every
$k\in\B N$ some conjugate of $g$ $k$-displaces $H$ then every element of $H'$ is
a product of seven $g$-commutators.  We get a better result under a stronger assumption.

\begin{proposition}\label{prop:bip}
Assume that $g\in\Gamma$ is such that for every finitely generated subgroup
$H<\Gamma$ and $k\in\B N$, there exists a conjugate of $g$ which $k$-displaces
$H$.  Then every element of\/ $\Gamma'$ is a product of two commutators in
$\Gamma$ and three $g$-commutators in $\Gamma$. Moreover,
$$
\Gamma'\subseteq\presp{\Gamma'}g\presp{\Gamma'}g^{-1}\presp{\Gamma'}g\presp{\Gamma'}g^{-1}\presp{\Gamma'}g\presp{\Gamma'}g^{-1}=\left(\presp{\Gamma'}g\presp{\Gamma'}g^{-1}\right)^3.
$$
\end{proposition}

\begin{proof}
Every element $f\in\Gamma'$ can be expressed as a product of $k$ commutators of
$2k$ elements of\/ $\Gamma$, for some $k\in\B N$.  Call the group they generate
$H$.  Since some conjugate of $g$, say $h$, $k$-displaces $H$, by Lemma
\ref{lem:bip}, there exist $\alpha$ $\beta$, and $\gamma\in\Gamma$ such that
$f=[\alpha,\beta][\gamma,h]$.

Since some conjugate of $g$ displaces the group generated by $\alpha$ and
$\beta$, by Lemma \ref{lem:commutators}, $[\alpha,\beta]$ is a product of two
$g$-commutators.  Thus $f$ is a product of three $g$-commutators.

The moreover part follows from the fact that $[\gamma,h]\in \presp{\Gamma}g\presp{\Gamma}g^{-1}$
and that  $\presp{\Gamma}h$ and $\presp{\Gamma'}h$ are equal.
The last claim follows from the fact that if 
$f\in\Gamma$ commutes up to comutation with $\presp gf^{-1}$ commute.  Then
$$
\presp{f}h
=\presp{f}{\left(\presp{\presp{g}f^{-1}}h\right)}=
\presp{[f,g]}h.
$$
\end{proof}

Note that the assumption of the above corollary implies that neither $\Gamma$
nor $\Gamma'$ are finitely generated.  However, we will use this approach to
prove uniform simplicity of the Higman-Thompson groups which are known to be
finitely generated.

\begin{lemma}\label{lem:perfect}
Assume that every two elements in $\Gamma$ commute up to conjugation.
Then every commutator in $\Gamma$ can be expressed as a commutator in $\Gamma'$.
In particular, $\Gamma'=\Gamma''$ is perfect.
\end{lemma}

\begin{proof}
Let $\alpha$ and $\beta$ belong to $\Gamma$.  Choose $h$ and $g$ such that
$\alpha$ and $\presp h\beta$ commute and also $\presp g\alpha$ and $[\beta,h]$ commute.
Then,
$
[[\alpha,g],[\beta,h]]=[\alpha,[\beta,h]]=[\alpha,\beta].
$
\end{proof}

\begin{proposition}\label{prop:bip2}
Let $g\in\Gamma'$ displaces $\Gamma_0<\Gamma$.  Assume that, for every $k\in\B N$,
every finitely generated subgroup $H<\Gamma_0$ is $k$-displaceable in
$\Gamma_0$.  Then, every element of\/ $\Gamma_0'$ is a product of four
$g$-commutators from $\Gamma'_0$. In particular $\Gamma_0'\subseteq
\left(\presp{\Gamma'_0}g\presp{\Gamma'_0}g^{-1}\right)^4$.
\end{proposition}

\begin{proof}
By Lemma \ref{lem:bip} every element of\/ $\Gamma_0'$ is a product of two commutators of\/ $\Gamma_0$.
By Lemma \ref{lem:perfect} they can be chosen to be commutators of elements of\/ $\Gamma_0'$.
By Lemma \ref{lem:commutators} each of them is a product of two $g$-commutators over $\Gamma'_0$.
\end{proof}

\section{Bounded actions on ordered sets}\label{sec:order}

The purpose of this section is to prove that numerous simple Higman-Thompson groups acting as order preserving piecewise-linear transformations are, in fact, uniformly simple.

We always assume that a group $\Gamma$ acts faithfully on the left
by order preserving transformations on a linearly ordered set $(I,\leq)$.
Given a map $g\colon I\to I$, we define the support $\supp(g)$ of $g$ to be
$\{x\in I : g(x)\neq x\}$.  Given $a$ and $b\in I$ we define
$(a,b) = \{y\in I : a<y<b\}$.  By $(a,\infty)$
we will denote the set $\{x\in I : a<x\}$. The group of all bounded
automorphisms of $(I,\leq)$ is denoted by $B(I,\leq)$.

We call such an action
\begin{itemize}
\item \textbf{proximal}, if for every $a,b,c,d\in I$, such that $a<b$ and
$c<d$ there is $g\in \Gamma$ satisfying $g(a,b)\supseteq (c,d)$;
\item \textbf{bounded}, if for every $g\in \Gamma$ there are $a,b\in I$,
such that $\supp(g)\subseteq (a,b)$.
\end{itemize}

Note, that being proximal implies that $(I,\leq)$ is dense without endpoints.

\begin{theorem} \label{thm:order}
Assume that $\Gamma$ acts faithfully, order preserving, boundedly, and proximally on
a linearly ordered set $(I,\leq)$.  Then its commutator group $\Gamma'$ is
\bound-uniformly simple and the commutator width of $\Gamma'$ is at most two.
\end{theorem}

\begin{proof}
We apply Proposition \ref{prop:bip}. Let $g$ be an arbitrary nontrivial element of
$\Gamma'$. Let $a\in I$ be such that $g(a)\neq a$.  Replacing $g$ by $g^{-1}$
we may assume that $a<g(a)$. Choose $b\in I$ such that $a<b<g(a)$.  Then
$g(a,b)\cap(a,b)=\varnothing$. Let $H$ be an arbitrary finitely generated
subgroup of\/ $\Gamma$.  Then, there exists an interval, say $(c,d)$, containing
supports of all generators of $H$, hence also containing supports of all
elements of $H$. By the proximality of the action, we may assume (possibly
conjugating $g$), that $(c,d)\subseteq (a,b)$. It is clear that such conjugate
of $g$ $\infty$-displaces $H$.  Thus Proposition \ref{prop:bip} applies.
\end{proof}

Let us apply Theorem \ref{thm:order} to the Higman-Thompson groups of order preserving piecewise-linear maps. We first recall the definitions. Let $q>r\geq 1$ be integers.
Recall that $\Fh_{q,r}$ ($\Fh_{q}$ respectively) is defined
as piecewise affine (we allow only finitely many pieces), order preserving
bijections of $\left((0,r)\cap\Zq,\leq\right)$ ($\left(\Zq,\leq\right)$ respectively)
whose breaking points of the derivatives belong to $\Zq$ and the slopes are $q^k$,
for $k\in\B Z$ (see the bottom of page 53 and the top of page 56 in \cite{MR885095}).

Define $\mathrm{B}\Fh_{q,r}$ ($\mathrm{B}\Fh_{q}$ respectively) to be the subgroup of $\Fh_{q,r}$
($\Fh_{q}$ respectively) consisting of all such transformations $\gamma$ that are boundedly
supported, that is, $\supp(\gamma)\subseteq (x,y)$, for some $x,y\in (0,r)\cap\Zq$
($x,y\in \Zq$ respectively).

We use the following lemma below. The first part of it is a known result \cite{MR3560537}.

\begin{lemma}\label{lem:ht}\ 
\begin{enumerate}
\item The groups $\mathrm{B}\Fh_{q,r}$ and $\mathrm{B}\Fh_q$ are isomorphic (\cite[Proposition \textSC{C}10.1]{MR3560537}).
\item The commutator subgroups of $\Fh_{q,r}$ and $\mathrm{B}\Fh_{q,r}$ are equal.
\end{enumerate}
\end{lemma}

\begin{proof}
%(1) Fix a biinfinite increasing sequence
%$(b_t)_{t\in\B Z}\subset (0,r)\cap\Zq$ such that $\lim_{t\to+\infty} b_t
%= r$, $\lim_{t\to-\infty} b_t = 0$ and $b_{t+1}-b_t\in \{q^s : s\in \B Z\}$.
%Define $\psi\colon\Zq\to(0,r)\cap\Zq$ sending affinely
%$[t,t+1]\cap\Zq$ onto $[b_t,b_{t+1}]\cap\Zq$.  Then
%\[\psi^*\colon\mathrm{B}\Fh_q \to\mathrm{B}\Fh_{q,r},\ \ \psi^*(x) = \psi x \psi^{-1}\] is an
%isomorphism, as $\psi$ has slopes $q^k$, for $k\in\B Z$.  Thus (1) holds.
%
(2) It is obvious that $\mathrm{B}\Fh_{q,r}'\subseteq \Fh_{q,r}'$. Let us prove
$\supseteq$. Note that $\Fh_{q,r}'\subseteq\mathrm{B}\Fh_{q,r}$ (because for
$g_1,g_2\in \Fh_{q,r}$, the element $[g_1,g_2]$ acts as the identity in some
small neighbourhoods of $0$ and $r$).  Thus, if $f\in \Fh_{q,r}'$, then
$\supp(f)\subseteq(b_{-j},b_j)$, for some $j\in\B Z$.  Therefore
$f(b_{-j},b_j)=(b_{-j},b_j)$. A slight modification of $\psi$ above gives
$\psi_j\colon\Zq\to(0,r)\cap\Zq$ which 

\begin{itemize}
\item sends $(-\infty,b_{-j}]\cap\Zq$ piecewise affinely onto $(0,b_{-j}]\cap\Zq$, 
\item is the identity on $[b_{-j},b_j]\cap\Zq$,
\item sends $[b_{j},+\infty)\cap\Zq$ piecewise affinely onto $[b_j,r)\cap\Zq$.
\end{itemize}
Then $\psi_j^*(x) = \psi_jx\psi_j^{-1}$ is another isomorphism between $\mathrm{B}\Fh_q$
and $\mathrm{B}\Fh_{q,r}$, such that $\psi_j^*(f)=f$ (we regard $\Fh_{q,r}$ as a
subgroup of $\mathrm{B}\Fh_q$). Write $f=\prod_{i=1}^m[g_{2i-1},g_{2i}]$, for $g_i\in
\Fh_{q,r}\subset\mathrm{B}\Fh_q$.  Then
$f=\psi_j^*(f)=\prod_{i=1}^m[\psi_j^*(g_{2i-1}),\psi_j^*(g_{2i})]\in
\mathrm{B}\Fh_{q,r}'$.
\end{proof}

We consider the action of $\mathrm{B}\Fh_q$ on $\Zq$ and its orbits. Let $I\lhd\Zq$
be the ideal of $\Zq$ generated by $(q-1)$.

\begin{lemma}[{\cite[Theorem \textSC{A}4.1, Corollary \textSC{a}5.1]{MR3560537}}] \label{lem:id}\ 
\begin{enumerate}
\item $I$ is $\mathrm{B}\Fh_q$-invariant.
\item $\mathrm{B}\Fh_q$ acts in a doubly-transitive way on $I$. In particular, the action is proximal.
\end{enumerate}
\end{lemma}

%\begin{proof}
%(1) Fix $a\in\Zq$ and $g\in\mathrm{B}\Fh_q$. By boundedness, there is $b\in\Zq$
%such that $g(b)=b$.  Divide $(a,b)$ into $N$ intervals of lengths of the
%form $q^k$ (i.e.  congurent to one $\mod (q-1)$) in such a way that the breaking
%points of $g$ are among the ends of those intervals.  Then
%$$
%a-b\equiv N\equiv g(a)-g(b)=g(a)-b \mod (q-1),
%$$
%and thus $a\in I$ if and only if $g(a)\in I$.
%
%(2) Given $a<b$ and $c<d$ with $a,b,c$, and $d\in I$. Choose $x$ and $y$ in $I$
%such that $x<\min\{a,c\}$ and $\max\{b,d\}<y$.  Since $(a-x)$, $(c-x)$,
%$(b-a)$, $(d-c)$, $(y-b)$, and $(y-d)$ are from $I$, we can, by taking iterated
%subdivisions, divide all those intervals in the same number of intervals, whose
%lengths are powers of $q$.  Then the piecewise affine map sending
%$a$ to $c$, $b$ to $d$ and which is the identity outside $(x,y)$,
%witnesses double-transitivity.
%\end{proof}

As a corollary of the above lemmata we get that groups $\Fh_{q,r}$ satisfy the assumptions
of Theorem \ref{thm:doubly-transitive}.

\begin{corollary}
$\Fh_{q,r}'\cong\mathrm{B}\Fh_q'$ is \bound-uniformly simple and the commutator width of it is at most two.
\end{corollary}

\begin{remark} \label{rem:ex}
Theorem \ref{thm:order} applies to the following groups.
\begin{itemize}
\item Bieri and Strebel \cite{MR3560537} define more general class of groups
acting boundedly on $\B R$.  They take a subgroup $P$ in the multiplicative
group $\B R_{>0}$ and a $\B Z[P]$-submodule $A<\B R$ and define $\Gamma:=B(\B
R;A,P)$ to be a group of boundedly supported automorphisms of $\B R$ consisting
of piecewise affine maps with slopes in $P$ and singularities in $A$.  They
define an augmentation ideal $I=\langle p-1|p\in P\rangle$ of $\B Z[P]$ and
prove that $\Gamma$ acts highly transitive on $IA$.  Thus $\Gamma'$ is \bound-uniformly
simple.
\item Another example of doubly-transitive and bounded action on a linear order
(thus satisfying the assumptions of Theorem \ref{thm:order}) was considered by
Chehata in \cite{MR0047031}, who studied partially affine transformations of an
ordered field and proved that this group is simple.  Theorem \ref{thm:order}
implies that the Chehata group is \bound-uniformly simple.
\end{itemize}
\end{remark}

\section{Proximality, primitivity, and double-transitivity} \label{sec:para}

In this section we prove (Theorem \ref{thm:compa}) that proximality (from the previous Section) and order-primitivity are equivalent properties for
bounded group actions. In general, these properties are inequivalent.  The action of
the group of integers on itself is primitive but neither proximal nor bounded. We also give an example of bounded, transitive and proximal action, which is not doubly-transitive (Theorem \ref{ex:proximal-not-2-transitive}).

An action of a group $\Gamma$ on a linearly ordered set $(I,\leq)$ is called
\textbf{primitive} (or \textit{order-primitive} by some authors), 
if for any other linearly ordered set $(J,\leq)$ and homomorphism
$\Psi\colon\Gamma \to\aut(J,\leq)$ and order preserving
equivariant map $\psi\colon(I,\leq)\to(J,\leq)$ (that is $\psi(\gamma x)=\Psi(\gamma)\psi(x)$),
the map $\psi$ is injective or $\psi(I)$ is a singleton.

\begin{theorem} \label{thm:compa}
Every proximal action is primitive.  Any bounded and primitive action is proximal.
\end{theorem}
\begin{proof}
Assume the action is not primitive.  Choose $a$, $b$ and $d$ such that $a\neq
b$ and $\psi(a)=\psi(b)\neq\psi(d)$.  Reversing the order if necessary, we may
assume $\psi(b)<\psi(d)$.  Set $c=a$.  This choice contradicts proximality,
as if $g(b,a)\subseteq (d,c)$ then
$$
\psi(d)\leq\Psi(g)\psi(b)=\Psi(g)\psi(a)\leq\psi(c)=\psi(b)<\psi(d).
$$

Assume that action is bounded, but not proximal. Let $a$, $b$, $c$, and $d$
witness the latter. For $x,y\in I$, $x<y$ consider the relation $\sim_{x,y}$ on $I$ defined as
\begin{quote}
$s\sim_{x,y}t$ if $s\leq t$ and there is no $\gamma\in\Gamma$ such that $\gamma(s,t)\supseteq(x,y)$.
\end{quote}
By the assumption $a\sim_{c,d} b$.  Let $\approx_{c,d}$ be the transitive
closure of $\sim_{c,d}$.  The symmetric closure $\simeq_{c,d}$ of
$\approx_{c,d}$ is transitively closed, thus $\simeq_{c,d}$ is an equivalence
relation, which has convex classes. Moreover, $\simeq_{c,d}$ is
$\Gamma$-invariant, that is $x\simeq_{c,d}y$ implies $\gamma
(x)\simeq_{c,d}\gamma(y)$ for all $\gamma\in\Gamma$. It is enough to prove that
$\simeq_{c,d}$ is not total, that is, $e\not\simeq_{c,d} f$ for some $e,f\in
I$, because then the quotient map \[\psi\colon I\to I/\simeq_{c,d}\] proves
nonprimitivity of the action ($I/\simeq_{c,d}$ has a natural $\Gamma$-action).

First, we claim that there is $\gamma\in\Gamma$ such that $\gamma(c)\geq d$.
Indeed, if there is no such group element, define a map $\psi\colon
I\to\{0,1\}$ by the formula
$$
\psi(x)=
\begin{cases}
0&\text{there is no $\gamma\in\Gamma$ such that $\gamma(x)\geq d$},\\
1&\text{there is $\gamma\in\Gamma$ such that $\gamma(x)\geq d$}.\\
\end{cases}
$$
This map would contradict primitivity.

Choose $e$ and $f$ from $I$ such that $\supp(\gamma)\subseteq (e,f)$.  Then
$\left\{\gamma^t(c,d) : t\in\B Z\right\}$ is a countable family of intervals in
$(e,f)$, which are pairwise disjoint. We claim that $e\not\simeq_{c,d} f$, as
otherwise there are $x,y\in [e,f]$, $x< y$ such that $x\sim_{c,d}y$ and $(x,y)$
contains $\gamma^t(c,d)$ for some $t\in\B Z$, which is impossible. 
%
%Obviously $\approx_{c,d}$ is finer than $\sim_{e,f}$ since $(e,f)$ contains infinitely many translates of $(c,d)$ by powers of $\gamma$.  Thus $\approx_{c,d}$ is strictly finer than $<$.  Thus $\simeq_{c,d}$ is a nontrivial proper equivalence relation.
\end{proof}

Clearly, if $\Gamma$ acts proximally on $(I,\leq)$, then in acts in a such way on any orbit.
Thus, we will restrict to transitive actions.

Examples of actions we discuss above are doubly-transitive (cf.~Lemma \ref{lem:id}(2) and
Remark \ref{rem:ex}).  Thus they are proximal.  This property seems to be
easier to check than doubly-transitivity. We construct below an example of
bounded, transitive and proximal action, which is not double-transitive.  The
reader may consult this result with a result of Holland \cite[Theorem
4]{MR0178052}, which says that every bounded, transitive, primitive and closed under $\min$, $\max$ action 
 must be doubly-transitive. Moreover, any
group acting boundedly and transitively cannot be finitely generated.
Indeed, finite number of elements have supports in a common bounded interval,
thus the whole group is supported in that interval, so does not act transitively.

\begin{theorem}\label{ex:proximal-not-2-transitive}
There exists a subgroup $\Gamma<B(\B Q,\leq)$ acting transitively and proximally
but not doubly-transitively.
\end{theorem}

\begin{proof}
For each $k\in\B N$ we will define a countable linear order $(I_k,\leq)$,
a group $\Gamma_k$ acting on it, and a function $f_k\colon I_k\times I_k\to\B Z$
such that:
\begin{enumerate}
\item $\Gamma_k < \Gamma_{k+1}$;
\item $I_k$ is a $\Gamma_k$-equivariant linear bounded suborder of $I_{k+1}$;
\item for $k>0$, $\Gamma_k$ acts transitively and proximally on $I_k$ by order preserving transformations (but not doubly-transitive);
\item $f_k$ is $\Gamma_k$-invariant: $f_k(\gamma a,\gamma b)=f_k(a,b)$, for $\gamma\in\Gamma_k$, $a,b\in I_k$ and $f_k\subset f_{k+1}$.
\end{enumerate}

Then we take $\Gamma_\infty = \bigcup_{k\in\B N}\Gamma_k$, which acts boundedly,
transitive and proximally, but not doubly-transitive on
$I_\infty = \bigcup_{k\in\B N}I_k$, because of
$f_\infty = \bigcup_{k\in\B N} f_k$, which is a $\Gamma_\infty$-invariant map $I_\infty\times I_\infty\to \B Z$.

%Since $(I_\infty,\leq)$ is (by proximality) a countable dense linear order without ends it is isomorphic to $(\B Q,\leq)$.

Since $(I_\infty,\leq)$ is a countable and, by proximality, dense linear order
without ends, it is isomorphic to $(\B Q,\leq)$.

In the following inductive construction we will define three auxiliary points
$i_k^-<i_k<i_k^+$ from $I_k$.

We put $\Gamma_0=\B Z$ and $I_0=\B Z$, where $\Gamma_0$ acts on $I_0$ by translations. Let $f_0(n,m)=n-m$ and $i_0^-=-1$, $i_0=0$, $i_0^+=1$.

Assume we have constructed $I_k$, $\Gamma_k$, and $f_k$.
Let $I_{k+1} = \left\{a\in {I_k}^{\B Z} : \forall^\infty n\in \B Z\ a(n)=i_k \right\}$ and $i_{k+1}(n)=i_k$
for all $n\in\B Z$. In plain words, $I_{k+1}$ consists of all functions from $\B Z$ to $I_k$
which differ from a constant function (denoted by $i_{k+1}$) taking the value $i_k$,
only at finite many places.  Define a linear order on $I_{k+1}$ by putting $a<b$
if $\min\left\{n\in \B Z:a(n)<b(n)\right\}<\min\left\{n\in \B Z:a(n)<b(n)\right\}$,
with the convention that $\min\varnothing >n$ for all $n\in\B Z$.
Note that $I_k$ embeds into $I_{k+1}$:
$$
I_k \ni a\mapsto\left(n\mapsto\begin{cases}a&\text{if }n=0,\\i_k&\text{otherwise}\\\end{cases}\right)\in I_{k+1}.
$$

Consider $\conv(I_k)=\left\{a\in I_{k+1} : a(n)=i_k \text{ for all } n<0\right\}$, with the following action of\/ $\Gamma_k$:
$$
(\gamma a)(n) = \begin{cases}\gamma a(0)&\text{if }n=0,\\a(n)&\text{otherwise}\\\end{cases}.
$$
Define
$i_{k+1}^{\pm}(n)=\begin{cases}i_k^{\pm}&\text{if }n=-1,\\0&\text{otherwise}\\\end{cases}$.
The interval $(i_{k+1}^-,i_{k-1}^+)\subset I_{k+1}$ contains the embedded copy of $I_k$.

Extend the action of\/ $\Gamma_k$ to the whole of $I_{k+1}$ by the identity
on the complement $I_{k+1}\smallsetminus \conv(I_k)$.
Thus the action of\/ $\Gamma_k$ on $I_{k+1}$ is bounded.
Define yet another automorphism $\sigma_{k+1}$ of $I_{k+1}$ by $(\sigma_{k+1}a)(n)=a(n+1)$.
Let $\Gamma_{k+1}$ to be the group generated by $\Gamma_k$ and $\sigma_{k+1}$.
The action of\/ $\Gamma_{k+1}$ on $I_{k+1}$ is clearly transitive.

For every pair $a\neq b$ from $I_{k+1}$, define $m_{a,b}=\min\{n\in\B Z : a(n)\neq b(n)\}$.

For $a<b$ and $c<d$ let $\gamma\in\Gamma_k$ be such that
$(c(m_{c,d}),d(m_{c,d}))\subseteq\gamma(a(m_{a,b}),b(m_{a,b}))$
(such $\gamma$ exists by proximality of the action of $\Gamma_k$ on $I_k$).  Then
$(c,d)\subseteq \sigma_{k+1}^{-m_{c,d}}\gamma\sigma_{k+1}^{m_{a,b}+1}(a,b)$,
which proves the proximality of the action of $\Gamma_{k+1}$ on $I_{k+1}$.

Finally, define
$f_{k+1}(a,b) = f_k(a(m_{a,b}),b(m_{a,b}))$.
Clearly, $f_{k+1}$ is $\Gamma_{k+1}$-invariant, hence the action of
$\Gamma_{k+1}$ on $I_{k+1}$ is not doubly-transitive.
\end{proof}

The element $\sigma_k\in\Gamma_k$ stabilizes $i_k$ and has unbounded orbits on $(i_k,\infty)\subset I_k$.
Thus the stabiliser of $i_\infty=\lim i_k$ has unbounded orbits on $(i_\infty,\infty)\subset I_\infty$.
This is enough to conclude that the action is proximal.

\begin{question}
Is there any transitive, proximal bounded action without the property that point stabilisers have unbounded orbits?
\end{question}

\section{Extremely proximal actions on a Cantor set and uniform simplicity}
\label{sec:Cantor}

The main goal of the present section is prove Theorem \ref{thm:C}, which gives a criterion for a group acting on a Cantor set to be \boundd-uniformly simple.

Let $C$ be a Cantor set. Assume that a discrete group $\Gamma$ acts on $C$ by homeomorphisms.
By the \textbf{topological full group} $\G \Gamma<\homeo(C)$ of\/ $\Gamma$ we define (see e.g. \cite{MR1710743})
$$
{\G \Gamma}=\left\{g\in\homeo(C):
\begin{varwidth}{0.5\textwidth}
for each $x\in C$ there exists a neighbourhood $U$ of $x$ and $\gamma\in\Gamma$
such that
$g|_U=\gamma|_U$
\end{varwidth}
\right\}.
$$

Through this section we assume that:
\begin{itemize}
\item group $\Gamma$ acts faithfully by homeomorphisms on a Cantor set $C$;
\item $\Gamma$ is a topological full group, i.e. $\Gamma=\G\Gamma$;
\item the action is \textbf{extremely proximal}, i.e. 
for any nonempty and proper clopen sets $V_1,V_2\subsetneq C$
there exists $g\in\Gamma$ such that $g(V_2)\subsetneq V_1$.
\end{itemize}

The second assumption is not hard to satisfy as $\G \Gamma=\G{ \G \Gamma}$.

\begin{theorem}\label{thm:C}
Assume that $\Gamma$ satisfies the above assumptions.
Then $\Gamma'$, the commutator subgroup of\/ $\Gamma$, is \boundd-uniformly simple.
The commutator width of\/ $\Gamma'$ is at most three. Therefore, if $\Gamma$ is perfect (i.e. $\Gamma'=\Gamma$), then $\Gamma$ is \boundd-uniformly simple.
\end{theorem}

Before proving \ref{thm:C}, we need a couple of auxiliary lemmata.

Suppose $x\in C$ and $h\in\Gamma$. By the Hausdorff property of $C$, if
$h(x)\neq x$, then there exists a clopen subset $U\subset C$ containing $x$,
such that $h(U)\cap U=\varnothing$.  In such a situation we define an element
$\tau_{h,U}\in\Gamma$ exchanging $U$ and $h(U)$:
$$
\tau_{h,U}(x)=
\begin{cases}
x&\mbox{if }x\not\in U\cup h(U),\\
h(x)&\mbox{if }x\in U,\\
h^{-1}(x)&\mbox{if }x\in h(U).\\
\end{cases}
$$
Such an element belongs to $\Gamma$, since $\Gamma=\G\Gamma$ is a topological full group. Observe that $\tau_{h,U}^2=\id$ and $f\tau_{h,U}
f^{-1}=\tau_{\presp{f}h,f(U)}$, for $f\in\Gamma$.

\begin{lemma}\label{lem:fragmentation}
Assume $\Gamma$ acts extremely proximally on a Cantor set $C$.
\begin{enumerate}
\item $\Gamma'$ acts extremely proximally on $C$.
\item For any nontrivial $f\in\Gamma$ and a proper clopen $V\subsetneq C$
there is $h\in\Gamma'$ such that $V\cap\presp{h}f(V)=\varnothing$.
\item Let $f,g\in\Gamma$ be nontrivial.  Then there is $h\in\Gamma'$ such
that $\presp{h}g.f$ is supported outside a clopen subset.
\end{enumerate}
\end{lemma}

\begin{proof}
(1) Let $U$ and $V$ be nonempty and proper clopen subsets of $C$.
Shrinking $U$, if necessary, we may assume that $U\cup V\neq C$ (that is,
we may always take $g\in \Gamma$ and $U_1=g(U)$, $V_1=g(V)$, such that
$U_1\cup V_1\neq C$; then $h(U_1)\subsetneq V_1$ implies $h^g(U)\subsetneq V$).
By extremal proximality, find elements $g_1$, $g_2$, $h_1$, and $h_2$ in $\Gamma$ such that
$g_1(U)\subsetneq C\smallsetminus (U\cup V)$, $g_2(U)\subsetneq C\smallsetminus (U\cup V\cup g_1(U))$,
$h_1(V)\subsetneq g_1(U)$, $h_2(U)\subsetneq C\smallsetminus (U\cup V\cup g_1(U))$.
Define $g=\tau_{g_2,U}\tau_{g_1,U}$ and $h=\tau_{h_2,U}\tau_{h_1,U}$.

It is straightforward to check that, since $U$, $g_1(U)$, and $g_2(U)$ are
pairwise disjoint, we have $g^3=1$ which is equivalent to
$$
g=\tau_{g_2,U}\tau_{g_1,U}=\left[\tau_{g_1,U}\tau_{g_2,U}\right].
$$
And similarly for $h$.  In particular, $g$ and $h$ belong to $\Gamma'$. Furthermore, $g^{-1}h(U)=g^{-1}h_1(U)\subsetneq g^{-1}g_1(V)=V$.

(2) Choose $U$ to be a nonempty clopen such that $f(U)\cap U=\varnothing$.
Choose, by (1), $h\in\Gamma'$ such that $h^{-1}(V)\subsetneq U$.  Then $V\cap\presp{h}f(V) \subseteq h(U\cap f(U))= \varnothing$.

(3) We may choose clopens $U$ and $V$ such that $f(U)\cap U=\varnothing=g(V)\cap V$.
If $h_1\in\Gamma'$ satisfies $h_1^{-1}(U)\subsetneq V$, then $\presp{h_1}g(U)\cap U=\varnothing$ (such a $h_1$ exists by (2)).

If $\presp{h_1}gf$ is the identity on $U$ the proof is finished. Otherwise
define $\gamma=\presp{h_1}g$.
We may find $W\subset U$ such that $\gamma f(W)\cap W=\varnothing$ and $W\cup f(W)\cup \gamma^{-1}(W)\subsetneq C$.
Notice that $\gamma^{-1}(W)$, $W$ and $f(W)$ are pairwise disjoint.

Choose $\eta\in\Gamma$ such that $\eta(W)\cap \left(W\cup f(W)\cup \gamma^{-1}(W)\right)=\varnothing$.
Put $\tau_1=\tau_{\eta\gamma,\gamma^{-1}W}$, $\tau_2=\tau_{f\gamma,\gamma^{-1}(W)}$ and $h_2=[\tau_1,\tau_2]$.
As in (1), we have that $h_2=\tau_1\tau_2\in\Gamma'$ and if $w\in W$, then
$h_2(w)=w$ and $h_2\gamma^{-1}(w)=\tau_1f^{-1}w=f^{-1}w$.

Hence $\presp{h_2h_1}gf=\presp{h_2}\gamma f$ is the identity on $W$.
Indeed, let $w\in W$.  Then $f(w)\in f(W)$.  Thus  $h_2^{-1}f(w)=\gamma^{-1}(w)$
i.e. $\gamma h_2^{-1}f(w)=w\in W$.  Therefore
$h_2\gamma h_2^{-1}f(w)=w$.
\end{proof}

For any clopen $U\subset C$, let $\Gamma_U$ be the subgroup of\/ $\Gamma$
consisting of elements of\/ $\Gamma$ supported on $U$.

\begin{lemma}\label{lem:commutant}
Let $V\subsetneq C$ be a proper clopen set.
Then there exists a proper clopen $V\subsetneq U\subsetneq C$ such that $\Gamma'\cap \Gamma_V\subset \Gamma_U'$.
\end{lemma}

\begin{proof}
Let $\alpha\in\Gamma$ be such that $\alpha(V)\supsetneq V$.  Let $U=V\cup \alpha(C\smallsetminus V)\subsetneq C$. Define $\psi\colon U\to C$  
\[
\psi(x)=
\begin{cases}
x&\mbox{if }x\in V,\\
\alpha^{-1}(x)&\mbox{if }x\in \alpha(C\smallsetminus V).\\
\end{cases}
\]
Then $\psi$ is a homeomorphism, which induces an isomorphism $\Psi\colon \Gamma\to\Gamma_U$ given by 
\[
\Psi(h)(x)=
\begin{cases}
x&\mbox{if }x\in C\smallsetminus U,\\
\psi^{-1}(h(\psi(x)))&\mbox{if }x\in U,\\
\end{cases}
\] for any $h\in\Gamma$ and $x\in C$. Since $\Psi$ is the identity on $\Gamma_V$, $\Psi(f)=f$, for any $f\in \Gamma_V$.  Therefore, if $f\in\Gamma'$, then $f\in \Gamma_U'$.
\end{proof}

\begin{lemma}\label{lem:bip-element}
Assume that $U\subsetneq V\subseteq C$ are clopens.  There exists
$h\in\Gamma_V'$ such that for all $k\in\B Z$, the sets $h^k(U)$ are pairwise
disjoint.
\end{lemma}

\begin{proof}
Choose clopen $W$ such that $U\subsetneq W\subsetneq V$.
By extremal proximality, choose $\beta$ and $\gamma\in\Gamma$ such that
$\beta(W)\subset V\smallsetminus W$ and $\gamma(W)\subset W\smallsetminus U$.
Define $\alpha\in \Gamma_V$ by
\[
\alpha (x)=
\begin{cases}
x&\mbox{if }x\in C\smallsetminus (W\cup\beta(W)),\\
\gamma^{-1}(x)&\mbox{if }x\in\gamma(W),\\
\beta(x)&\mbox{if }x\in W\smallsetminus \gamma(W),\\
\presp\beta\gamma(x)&\mbox{if }x\in\beta(W).\\
\end{cases}
\]
Then the sets $\alpha^k(U)$ are pairwise disjoint. Indeed, it is sufficient to
prove that $\alpha^k(U)\cap U=\varnothing$, for all $k>0$.  Since $U\subset
W\smallsetminus\gamma(W)$, we have $\alpha(U)\subset \beta(W)$.  As $\alpha
\beta(W)\subset \beta(W)$, for $k\geq 1$, $\alpha^k(U)\subset\beta(W)$ which is
disjoint from $U$.

Since $\tau_{\beta,W}\in\Gamma_V$ conjugates $\alpha$ to $\alpha^{-1}$,
the element $h=\alpha^2=[\alpha,\tau_{\beta,W}]$ satisfies the claim.
\end{proof}

\begin{proof}[Proof of Theorem \ref{thm:C}]
Let $f$ be an element of\/ $\Gamma'$ and $A$ be a nontrivial conjugacy class of
$\Gamma'$.  By Lemmata \ref{lem:fragmentation}(3) and \ref{lem:commutant} we
have that $f=g_1^{-1}f_1$ for some $g_1\in A$ and $f_1\in\Gamma_{V_1}'$ for
some proper clopen ${V_1}\subsetneq C$.  

%By Lemma \ref{lem:commutant}, $\Gamma'\cong\Gamma'_{V_1}$, for some proper clopen $V_1\subsetneq C$. We prove the conclusion for $\Gamma'_{V_1}$. Let $f_1\in\Gamma_{V_1}'$. Let $A$ be a nontrivial conjugacy class of\/ $\Gamma'_{V_1}$. 

We claim that $f_1$ is a product of four $A$-commutators in $\Gamma'$.
Choose $V_1\subsetneq V_0\subsetneq C$ and $\omega\in V_0\smallsetminus V_1$.
We apply Proposition \ref{prop:bip2}. Namely, let $\Gamma_0$ denote the union of
groups $\Gamma_V$, such that $V$ is a clopen contained in
$V_0\smallsetminus\{\omega\}$. Clearly, $\Gamma_0$ is a proper subgroup of
$\Gamma_{V_0}$. By Lemma \ref{lem:fragmentation}(2), we may choose $g\in A$,
such that $g (V_0)\cap V_0=\varnothing$.  Thus, $g$ displaces $\Gamma_0$. Let $H$
be a finitely generated subgroup of\/ $\Gamma_0$. The union of supports of its
generators is a clopen $U$, properly contained in $V_0$, since $\omega\not\in
U$. Hence $H<\Gamma_U<\Gamma_0$.  Choose $U\subsetneq V\subsetneq V_0$, such
that $\omega\not\in V$. Let $h\in\Gamma'_V<\Gamma'_0$ be as in Lemma
\ref{lem:bip-element}.  Then $h$ $\infty$-displaces $H$. Thus Proposition
\ref{prop:bip2} applies and $f_1\in\Gamma_{V_1}'<\Gamma_0'$ is a product of four
$g$-commutators.

By Lemma \ref{lem:bip}, the commutator width of\/ $\Gamma_0'$ is at most two.  By
Lemma \ref{lem:fragmentation}(3), every element decomposes as a product of a
conjugate of a given nontrivial element from $\Gamma'$, say a commutator, and an
element conjugate into $\Gamma_0'$. Thus every element of\/ $\Gamma'$ is a
product of three commutators.
\end{proof}

\section{Groups almost acting on trees}\label{sec:Neretin}

In this section we apply Theorem \ref{thm:C} to groups almost acting on trees.

By a \textbf{graph} (whose elements are called \textbf{vertices}) we mean a set,
equipped with a symmetric relation called adjacency.  A
\textbf{path} is a sequence of vertices indexed either by a set $\{1,\dots,n\}$
or $\B N$ (in such a case we call the path a \textbf{ray}) such that
consecutive vertices are adjacent, and no vertices whose indices differ by two
coincide (i.e. there are no backtracks).  A graph is called a \textbf{tree} if is
connected (nonempty) and has no cycles, i.e. paths of positive length starting and
ending at the same vertex (in particular, the adjacency relation is irreflexive).

\textbf{Ends of\/ $T$} are classes of infinite rays in $T$.  Two rays
are equivalent if they coincide except for some finite (not necessarily of the
same cardinality) subsets.  The set of all ends of\/ $T$ is denoted by $\partial
T$, and is called the \textbf{boundary of\/ $T$}.

Given a pair of adjacent vertices (called an \textbf{oriented edge}) ${\vec e}=(v,w)$,
we call the set of terminal vertices of paths starting at ${\vec e}$
a \textbf{halftree} of\/ $T$ and we will denote it by $T_{\vec e}$.  The classes of
rays starting at ${\vec e}$ will be called the end of a halftree
$T_{\vec e}$ and will be denoted by $\partial T_{\vec e} \subset \partial T$.  By $-\vec e$ we denote
the pair $(w,v)$.

We endow $\partial T$ with a topology, where the basis of open sets consist of
ends of all halftrees.

A \textbf{valency} of a vertex $v$ is the cardinality of the set of vertices adjacent to $v$.
A vertex of valency one is called a \textbf{leaf}.  If every vertex has valency
at least three but finite, then the boundary $\partial T$ is easily seen to be compact,
totally disconnected, without isolated points,
and metrizable. Thus, $\partial T$ is a Cantor set.
In such a case, every end $\partial T_{\vec e}$\ \  of a halftree is a clopen (open and closed) subset of $\partial T$.
% Every clopen subset of $\partial T$ is a disjoint union of finitely many ends of halftrees.

A \textbf{spheromorphism} is a class of permutations of\/ $T$ which preserve all
but finitely many adjacency (and nonadjacency) relations.  Two such maps are
equivalent if they differ on a finite set of vertices (see e.g. \cite[Section
3]{1502.00991}).  We denote the group of all spheromorphisms of\/ $T$ by
$\aaut(T)$.  If\/ $T$ is infinite, then the natural map $\aut(T)\to\aaut(T)$ is
an embedding.
Every sphereomorphism $f\in\aaut(T)$ induces a homeomorphism of
its boundary $\partial T$. 

For an integer $q>1$, by $T_q$ we denote the regular tree whose vertices have
degree $(q+1)$.  The group $\Nr_q$ was introduced by Neretin in \cite[4.5,
3.4]{MR1209033} as the group $\aaut(T_q)$ of sphereomorphisms of
$(q+1)$-regular tree $T_q$. It is abstractly simple \cite{MR1703086}.

In what follows, we will be interested in subgroups $\Gamma<\aut(T)$ acting
extremely proximally on the boundary $\partial T$ (see Theorem \ref{thm:parr}
and Corollary \ref{cor:free} below).  The whole
group of automorphisms $\Gamma=\aut(T_q)$ of\/ $T_q$ is such an example.  Another
example (cf. Example \ref{ex:cycki}) is the automorphism group $\Gamma=\aut(T_{s,t})$ of a bi-regular tree
$T_{s,t}$, $s,t>2$ (i.e. every vertex of\/ $T_{s,t}$ is black or white, every black
vertex is adjacent with $s$ white vertices, every white --- with $t$ black
vertices). We prove that the group $\G\Gamma$ of partial $\Gamma$-actions on
$\partial T$ is then \boundd-uniformly simple.

The group $\aut(T_{s,t})$ itself is virtually 8-uniformly simple \cite[Theorem
3.2]{MR3085032}. (Bounded simplicity in \cite{MR3085032} means uniform
simplicity in our context.)

There is a connection between the notion of a sphereomorphism
and topological full group acting on a boundary of a tree.

\begin{example}\  \label{ex:tree}
\begin{enumerate}
\item Any subdivision of $\partial T$ into clopens can be refined to $\C U_1$,
a subdivision into ends of halftrees (since any clopen in $\partial T$ is a
finite union of boundaries of halftrees).  Therefore the Neretin group $\Nr_q$ can
be characterized as $\Nr_q=\G{\aut(T_q)}=\aaut(T_q)$.

\item Another, well studied, example comes from considering 
$$
\aut_0(T_q)=\left\{
\begin{varwidth}{0.7\textwidth}
automorphisms of\/ $T_q$ preserving chosen cyclic orders on edges adjacent to any vertex of\/ $T_q$
\end{varwidth}
\right\}.
$$
One may induce cyclic orders by planar representation of\/ $T_q$.  The group
$\G {\aut_0(T_q)}$ is the Higman-Thompson group $\Th_{q,2}$ \cite[Section
5]{1502.00991}, \cite[2.2]{MR1703086}.

\item Those two examples can be generalized in the following manner (see
\cite[Section 3.2]{MR1839488}). Let $c\colon E(T_q)\to \{0,\ldots,q\}$ be a
function from the set $E(T_q)$ of (undirected) edges of $(q+1)$-regular tree $T_q$,
such that for every vertex $v$, the restriction of $c$ to the set of edges
$E(v)$ starting at $v$ gives a bijection with $\{0,\ldots,q\}$. We say that
such $c$ is a \textbf{proper colouring} of\/ $T_q$. Let $F<S_{q+1}$ be a
subgroup of permutations of $\{0,\ldots,q\}$. Using proper colouring $c$ and
$F$ we define the \textbf{universal group} $U(F)$ to be
$$
U(F) = \left\{g\in\aut(T_q) : c\circ g \circ c^{-1}_{|E(v)}\in F,
\text{ for every vertex }v\right\}.
$$
In fact $U(F)$ is independent (up to conjugation is $\aut(T_q)$) of the choice
of proper colouring $c$. We prove (see Corollary \ref{cor:ner}) that
$\G{U(F)}'$ is \boundd-uniformly simple, provided that $F$ is transitive on
$\{0,\ldots,q\}$. If $F$ is generated by a $(q+1)$-cycle, then
$U(F)=\aut_0(T_q)$, from (2). If $F=S_{q+1}$, then $U(F)=\aut(T_q)$.
\end{enumerate}
\end{example}

We call an action for a group $\Gamma$ on a tree $T$ \textbf{minimal} if there is no proper
$\Gamma$-invariant subtree of $T$.
Given a subset $A$ of a tree. We define its \textbf{convex hull} to be the set
of all vertices which lie on paths with both ends in the set $A$. It is a subtree.
The action is minimal if and only if the convex hull of any orbit is the whole tree.

\begin{example} \label{ex:finite}
Every action on a leafless tree with a finite quotient is minimal.
The converse is not true (see Example \ref{ex:cycki}).
\end{example}

Indeed, the distance from a $\Gamma$-orbit is a bounded function.  Hence the complement of
an orbit cannot contain an infinite ray.  Thus every vertex lies on a path with endpoints
in a given orbit.  

\begin{lemma}[{\cite[Lemma 4.1]{MR0299534}}]\label{lem:orbit}
Assume that a group $\Gamma$ acts minimally on a leafless tree $T$.
Then for every vertex $v$ and an edge $\vec e$ the orbit $\Gamma v$
intersects the halftree $T_{\vec e}$.
\end{lemma}

\begin{proof}
If $\Gamma v$ is all contained in $T_{-\vec e}$, so is its convex hull.  Thus the claim.
\end{proof}

We call an action for a group $\Gamma$ on a tree $T$ \textbf{parabolic} if $\Gamma$
has a fixed point in $\partial T$.

An action of a group by homeomorphisms on a topological space is called
\textbf{minimal} if there is no proper nonempty closed invariant set (equivalently,
if every orbit is dense).  This notion should not cause confusion with the notion
of minimal actions on trees. (A tree is a set equipped with a relation as opposed
to its geometric realisation which is a topological space.)

\begin{theorem} \label{thm:parr}
Assume that $T$ is a leafless tree such that $\partial T$ is a Cantor set.
Let $\Gamma$ act on~$T$.  The following are equivalent.
\begin{enumerate}
\item The action of\/ $\Gamma$ on $\partial T$ is extremely proximal
(see the beginning of Section \ref{sec:Cantor} for the definitions).
\item The action of\/ $\G\Gamma$ on $\partial T$ is extremely proximal.
\item The action of\/ $\Gamma$ on $\partial T$ is minimal and $\partial T$
does not support any $\Gamma$-invariant probability measure.
\item The action of\/ $\Gamma$ on $T$ is minimal and not parabolic,
that is, there is no proper $\Gamma$-invariant subtree of $T$ and $\Gamma$
has no fixed point in $\partial T$.
\end{enumerate}
\end{theorem}

\begin{proof}
($1\Rightarrow 2$)  This is straightforward.

($2\Rightarrow 3$)
Let $F$ be a closed, nonempty, proper and $\Gamma$-invariant subset of $\partial T$.
Choose $x\in F$ and a proper clopen $V\subset \partial T$ containing $x$.
Define $U=\partial T\smallsetminus F$. Then, there is no $g\in\G\Gamma$ such
that $g(V)\subset U$, since $g(x)=\gamma (x)\in F$, for some $\gamma\in\Gamma$;
thus contradiction.

Similarly, let $\mu$ be a $\Gamma$-invariant measure on $\partial T$.
Decompose $\partial T=U_1\cup U_2\cup U_3$, where $U_i$'s are disjoint nonempty clopens.
We may assume that $\mu(U_1)<\nicefrac12$.  Then, there is no
$g\in\G\Gamma$ such that $g(U_2\cup U_3)\subset U_1$.  Indeed, for any $g\in\G\Gamma$
we may decompose $U_2\cup U_3$ (by compactness) as finite disjoint union of clopens $U_2\cup U_3=\bigcup_{i=1}^k V_i$
such that $g|_{V_i}=\gamma_i|_{V_i}$ for some $\gamma_i\in\Gamma$ and then
$$
\nicefrac12<\mu(U_2\cup U_3)=\sum_{i=1}^k\mu(V_i)=\sum_{i=1}^k\mu(\gamma_iV_i)<\mu(U_1)<\nicefrac12
$$
is a contradiction.  Hence, the action is not extremely proximal.

($3\Rightarrow 4$)  If there is an infinite $\Gamma$-invariant subtree $T'$ of $T$ or a fixed point $\omega\in\partial T$,
then either $\partial T'$ or $\{\omega\}$ is a $\Gamma$-invariant closed subset of $\partial T$.

Suppose that there exists a finite $\Gamma$-invariant subtree $T'$ of $T$. We
use the following definition. Given a vertex $v$ of $T$, we define the \textbf{visual measure}
associated to $v$ to be the unique measure $\mu_v$ on $\partial T$ with the following property: if
$\{v_i\}_{i=0}^n$ is any injective path starting at $v_0=v$, then
$$
\mu_v\left(\partial T_{(v_{n-1},v_n)}\right)=\frac{1}{d_0\prod_{i=1}^{n-1}(d_i-1)},
$$
where $d_i$ is the valence of $v_i$.  The visual metric $\mu_v$ is
obviously invariant under the action of the stabiliser $\stab(v)$ of $v$ in $\aut(T)$.

We can consider the average of the visual measures
associated to the vertices of this subtree $T'$. It will be a $\Gamma$-invariant measure on $\partial T$.

($4\Rightarrow 1$) By Lemma \ref{lem:orbit} we may assume that, for every pair
of edges $\vec e$ and $\vec f$, there is $\gamma\in\Gamma$ such that either
$T_{\gamma\vec e}$ or $T_{-\gamma\vec e}$ is strictly contained in $T_{\vec f}$.
It is enough to show that one can find $\gamma\in\Gamma$, such
that the later holds, i.e. $\partial T_{-\gamma(\vec e)}\subsetneq \partial T_{\vec f}$ (indeed, since ends of halftrees constitute a basis,
we can find edges $\vec e$ and $\vec f$ such that $\partial T_{\vec e}\subset U$
and $\partial T_{\vec f}\subset C\smallsetminus V$, for nonempty proper clopens $V$ and $U$ in $\partial T$; if there is $\gamma\in\Gamma$ such that $\partial T_{-\gamma(\vec e)}\subsetneq \partial T_{\vec f}$ then $\gamma V\subseteq\partial T_{-\gamma(\vec e)}\subsetneq\partial T_{\vec f}\subseteq U$).

It is enough to prove this claim for $\vec e=\vec f$.
Indeed, if there exists $\gamma_1\in\Gamma$ such that $T_{\gamma_1\vec e}\subsetneq T_{\vec f}$ 
and $T_{-\gamma_2\vec e}\subsetneq T_{\vec e}$, then
$T_{-\gamma_1\gamma_2\vec e}\subsetneq T_{\gamma_1\vec e}\subsetneq T_{\vec f}$.

Assume that there exists $\gamma\in\Gamma$ such that $T_{\gamma\vec e}\subsetneq T_{\vec e}$.
Let $\{v_i\}_{i=0}^n$ be a path such that $\vec e=(v_0,v_1)$ and $\gamma\vec e=(v_{n-1},v_n)$.  Then
$\{v_i\}_{i\in\B Z}$, defined as $v_{nq+r}=\gamma^qv_r$, is a biinfinite path.  Let $\omega$ be its
end as $i\to\infty$.  Choose $\eta\in\Gamma$ such that $\eta(\omega)\neq\omega$. Consider the biinfinite
path from $\omega$ to $\eta(\omega)$. It coincides with $\{v_i\}_{i<i_-}$ and $\{\eta v_{-i}\}_{i>i_+}$
for some $i_\pm\in\B Z$.  Therefore $T_{-\eta\gamma^k\vec e}\subsetneq T_{\gamma^k\vec e}$, for $k$
big enough. Hence, $T_{-\gamma^{-k}\eta\gamma^k\vec e}\subsetneq T_{\vec e}$.  Thus the claim.
\end{proof}

\begin{remark}
Only clause (3) from \ref{thm:parr} concerns an action of a group on a tree.
The other parts of \ref{thm:parr} are about actions on a Cantor set. We do not
know if there is a straight argument for proving equivalence of $(1)$ and $(4)$
from Theorem \ref{thm:parr}, without referring to actions on trees.
\end{remark}

Below is an application of Theorems \ref{thm:C} and \ref{thm:parr} to the Neretin groups and the Higman-Thompson groups.

\begin{corollary} \label{cor:ner}\
\begin{enumerate}
\item Suppose $F<S_{q+1}$ is a transitive permutation subgroup and let  $c$ be a
proper colouring of\/ $T_q$ (see Example \ref{ex:tree}(3)).  Then $U(F)$ acts
transitively on the directed edges of\/  $T_q$, thus $\G {U(F)}'$ is
\boundd-uniformly simple.
\item Fix natural numbers $q>r\geq 1$.  The commutator subgroup $\Nr'_q$ of the Neretin group $\Nr_q$, and the Higman-Thompson group $\Th_{q,r}'$, 
are \boundd-uniformly simple and have commutator width bounded by three.
\end{enumerate}
\end{corollary}

\begin{proof}
Let $\Gamma=U(F)$.  Then the action of $\Gamma$ on $T_{q}$ 
is not parabolic as there is no $\stab(v)$-fixed
edge adjacent to $v$, hence no $\stab(v)$-fixed ray.
It is minimal since the action is transitive.

Therefore, in case of the Neretin group $\Nr'_q$ and the Higman-Thompson group $\Th_{q,2}'$,
Theorem \ref{thm:C} applies immediately due to Theorem \ref{thm:parr}.

Suppose $\C F$ is a family of pairwise disjoint ends of halftrees
$\partial T_{\vec e_i}\subset\partial T_q$, for $0\leq i\leq q-r$.
If $\Gamma_{\C F}$ is a pointwise stabiliser of $\C F$ in
$\G {\aut_0(T_q)}$ (see Example \ref{ex:tree}(2)), then $\Gamma_{\C F}$
is isomorphic to $\Th_{q,r}$ \cite[Section 5]{1502.00991}. Moreover,
$\Gamma_{\C F}$ is its own topological full group acting extremely proximally
on $C=\partial T_q\smallsetminus\bigcup_{i=0}^{q-r} \partial T_{\vec e_i}$.
Hence we get the conclusion for $\Th_{q,r}'$.
\end{proof}

\begin{corollary} \label{cor:free}
Suppose $\Gamma={\mathbf F_n}$ is a free group of rank $n\geq 2$.  Then $\Gamma$ acts
on its Cayley graph, which is $T_{2n-1}$.  This action is transitive and clearly
not parabolic.  Thus the induced action on the boundary is extremely proximal.
Therefore ${\G {{\mathbf F_n}}}'$ is \boundd-uniformly simple
by Theorem \ref{thm:C}.
\end{corollary}

\begin{example}[{\cite[Section 5]{MR0299534}, \cite[p. 232]{MR3085032}}]\label{ex:cycki}
We apply our results to trees constructed by Tits. Any connected graph $(G,E)$ of finite valence, with at least one edge, can appear as a quotient of a (finite valence) tree.

Assume that $c$ is a function from oriented edges of $G$ into the set of positive
integers.  By a result of Tits, there is a tree $T$ and a group $\Gamma$ acting on $T$
such that $G=\Gamma\backslash T$ and, for any $v$, and $w\in T$ such that
$(\Gamma v,\Gamma w)$ is an edge in $G$, there are exactly $c(\Gamma v,\Gamma w)$
vertices in $\Gamma w$ adjacent to $v$ (or none if it is not an edge of $G$).

If $c$ is such that the sum over edges starting at a given vertex is at least three (but finite), then 
the boundary of $T$ is a Cantor set.

If values of $c$ are at least two, the group action of $\Gamma$ on $T$ is
minimal and not parabolic \cite[5.7]{MR0299534},
i.e. the action of $\Gamma$ on $\partial T$ is extremely proximal due to Theorem \ref{thm:parr},
and $\G {\Gamma}$ is \boundd-uniformly simple due to Theorem \ref{thm:C}.
\end{example}

\begin{corollary} \label{cor:nir}
The groups of quasi-isometries and almost-isometries of a regular tree $T_q$ are five-uniformly simple.
\end{corollary}

\begin{proof}
This follows from Lazarovich results from the appendix.
Let $\Gamma$ be one of those groups.  By Theorem \ref{thm:nir-perfect} $\Gamma=\Gamma'$.
Since $\aut(T_q)$ is a subgroup of $\Gamma$, it acts extremely proximally on $\partial T_q$
(see Lemma \ref{lem:faithfull}) as a topological full group (see Lemma \ref{lem:qi-is-full}).  This already proves \boundd-uniform simplicity.

Let $1\neq g$ and $f$ be two elements of $\Gamma$. By Lemma \ref{lem:fragmentation} there
exists $g_1$, a conjugate of $g$, such that $f_1=g_1^{-1}f$ fixes a clopen in $\partial T_q$.
By Lemma \ref{lem:nir-commutator}, $f_1$ is a commutator of two elements fixing an open
set in $\partial T_q$.  Thus, by Lemma \ref{lem:commutators}, $f_1$ is a product
of two $g$-commutators.
\end{proof}

\newcommand{\QI}{\mathrm{QI}}
\newcommand{\AI}{\mathrm{AI}}

\section{Appendix by Nir Lazarovich: Simplicity of $\AI(T_{q})$ and $\QI(T_{q})$} \label{sec:app}

We begin by recalling the following definitions.

For $\lambda\geq 1$, and $K\geq 0$, a \textbf{$(\lambda,K)$-quasi-isometry} between
two metric spaces $(X,d_X)$ and $(Y,d_Y)$ is a map $f\colon X\to Y$ such that for all
$x,x' \in X$, \[\lambda ^{-1}d_X(x,x')-K \leq d_Y(f(x),f(x'))\le\lambda
d_X(x,x')+K,\] and for all $y\in Y$ there exists $x\in X$ such that
$d_Y(y,f(x))\leq K$.

A \textbf{$K$-almost-isometry} is a $(1,K)$-quasi-isometry.

A map $f$ is a \textbf{quasi-isometry} (resp. \textbf{almost-isometry}) if
there exist $K$ and $\lambda$ (resp. $K$) for which it is a
$(\lambda,K)$-quasi-isometry (resp. $K$-almost-isometry).

Two quasi-isometries $f_1,f_2\colon X\to Y$ are \textbf{equivalent} if they are at
bounded distance (with respect to the supremum metric).

The group of all quasi-isometries (resp. almost-isometries) from a metric space
$X$ to itself, up to equivalence, is denoted by $\QI(X)$ (resp. $\AI(X)$).  Thus,
for $q\geq 2$, we have the following containments:
$$
\aut(T_q)\subset N_q \subset \AI(T_q)\subset \QI(T_q)\subset \homeo(\partial T_q).
$$

Where the last containment follows from the following lemma.
\begin{lemma}\label{lem:faithfull}
The group $\QI(T_q)$ acts faithfully on $\partial T_q$. 	
\end{lemma}

\begin{proof}
Let $g\in\QI(T_q)$ be a quasi-isometry. Let $v\in T_q$, and let
$x_1$, $x_2$, $x_3\in\partial T_q$ be three distinct points such that $v$ is the
median of $x_1$, $x_2$, $x_3$, that is, $v$ is the unique intersection of all three
(biinfinite) geodesics $x_1x_2$, $x_1x_3$, $x_2x_3$.  Then, by the stability of
quasi-geodesics in Gromov hyperbolic spaces \cite[Theorem 1.7]{MR1744486}, $gv$
is at bounded distance (which does not depend on the vertex $v$) from the
midpoint of $gx_1$, $gx_2$, $gx_3$.  This implies that if $g$ induces the identity
map at the boundary, then $g\sim\id$.
\end{proof}

In fact, the proof above is valid whenever the space $X$ is a proper geodesic
Gromov hyperbolic space $X$ which has a Gromov boundary of cardinality at least
three whose convex hull is at bounded distance from $X$ (e.g. any
non-elementary hyperbolic group).

For what follows, let $\Gamma$ be the group $\QI(T_q)$ or $\AI(T_q)$ for $q\geq 2$.

\begin{lemma}\label{lem:qi-is-full}
The group $\Gamma<\homeo(\partial T_q)$ is a topological full group.
\end{lemma}

\begin{proof}
Fix $g\in \G{\Gamma}$, and let
$\left\{\partial T_{\vec e_1},\ldots,\partial T_{\vec e_n}\right\}$ be a disjoint cover
of $\partial T$ such that $g|_{\partial T_{\vec e_i}}=\gamma_i|_{\partial T_{\vec e_i}}$
for some $\gamma_i\in\Gamma$. For each $1\leq i\leq n$ let $\vec e_{i,1},\ldots,\vec e_{i,m}$
be such that $\left\{\partial T_{\vec e_{i,1}},\ldots,\partial T_{\vec e_{i,m}}\right\}$
is a disjoint cover of $g T_{\vec e_i}$. We may assume, by changing each $\gamma_i$
on a bounded set, that $\gamma_i(T_{\vec e_i}) = \bigcup_{j=1}^m T_{\vec e_{i,j}}$.

Let us define
$$
\gamma(v)=
\begin{cases}
\gamma_i(v)&\text{for $v\in T_{\vec e_i}$, and}\\
v&\text{otherwise.}
\end{cases}
$$
It is clear that if $\gamma$ is in $\Gamma$, then it induces the element $g$ on the boundary.

Let $\lambda, K$ be the maximal quasi-isometry constants of $\gamma_i$, and let $M$
be the diameter of the bounded set $\{\vec e_1, \ldots, \vec e_n, \gamma\vec e_1,\ldots,\gamma\vec e_n\}$. 

We claim the following: for all $v,w\in T_q$,
$d(\gamma v,\gamma w)\leq \lambda d(v,w)+ (2K+M)$. Indeed, if $v,w$ are both
in some $T_{\vec e_i}$ or in $T_q \setminus \bigcup _{i=1}^nT_{\vec e_i}$,
then the inequality is obvious. If $v\in T_{\vec e_i}$ and $w \in T_{\vec e_j}$
for some $i\ne j$, then $d(v,w)=d(v,e_i)+d(e_i,e_j)+d(e_j,w)$ and therefore 
\begin{align*}
d(\gamma x,\gamma y)= & d(\gamma_i x,\gamma_i e_i) + d(\gamma_i e_i, \gamma_j e_j) + d(\gamma_j e_j, \gamma_j w)\\
\leq & \lambda d(\gamma v,\gamma e_i) + K + M + \lambda d(e_j,w) +K\\
\leq & \lambda d(x,y) + 2K + M.
\end{align*}
Similarly, one shows this inequality for $v\in T_{\vec e_i}$ and $w \in T_q
\setminus \bigcup _{i=1}^n T_{\vec e_i}$.

Furthermore, the element $\gamma'$, defined as
$$
\gamma'(v)=
\begin{cases}
\gamma_i ^{-1}(v)&\text{if $v\in\bigcup _{j=1}^m T_{\vec e_{i,j}}$, and}\\
v&\text{otherwise,}
\end{cases}
$$
satisfies that for all $v,w\in T_q$, $d(\gamma' v,\gamma' w)\leq \lambda' d(v,w)+ (2K'+M')$,
for the appropriate $\lambda'$, $K'$, and $M'$. Moreover, it is easy
to see that $\gamma\gamma'\sim \id \sim \gamma'\gamma$, from which we deduce that $\gamma$ is a
quasi-isometry. 
\end{proof}

\begin{lemma}\label{lem:nir-commutator}
Every element $g$ in $\Gamma$ that fixes an open set at the boundary is a commutator
of two elements fixing a common set at the boundary.
\end{lemma}

\begin{proof}
Let $\supp g\subset T_{\vec e}$.
Let $\{x_n\}_{n\in\B Z}$ be a biinfinite line geodesic contained in $T_{-\vec e}$
and such that $x_0$ is the starting point of $\vec e$.
	
Let $t\in\aut(T_q)$ be a translation along $\{x_n\}_{n\in\B Z}$, and let $f$ be
the function defined by
$$
f(v)=
\begin{cases}
\presp{t^n}g^{-1}(v)&\text{for $v\in t^n(T_{\vec e})$ and $n\geq 0$, and}\\
v&\text{elsewhere.}
\end{cases}
$$
The function $f$ is in $\Gamma$ since all the functions $\presp{t^n}g^{-1}$
have the same quasi-isometry constants and $[t,f]=\presp{t}f f^{-1}=g$.

Let $s$ be a 1-almost-isometry defined as 
$$
s(v)=
\begin{cases}
t(v)&\text{for $v\in t^{-2}T_{\vec e}$,}\\
t^{-1}(v)&\text{for $v\in t^{-1}T_{\vec e}$,}\\
v&\text{otherwise.}
\end{cases}
$$
Then we still have $[ts,f]=g$ as $s$ commutes with $f$. However both
$st$ and $f$ fix $t^{-1}T_{\vec e}$.  Thus the claim.
\end{proof}

\begin{theorem}\label{thm:nir-perfect}
The group $\Gamma$ is perfect and has commutator width at most 2.
\end{theorem}

\begin{proof}
It suffices to show that each element of\/ $\Gamma$ can be written as a product of two elements
of\/ $\Gamma$ which fix an open set at the boundary, as
both of them are single commutators by Lemma \ref{lem:nir-commutator}.
	
Let $1\ne g\in\Gamma$, there exists $\omega\in\partial T$ such that $g(\omega)\ne\omega$.
Let $T_{\vec e}$ be a halftree whose boundary contains $\omega$ and
for which $g\partial T_{\vec e}$ and $\partial T_{\vec e}$ are
disjoint, and do not cover the whole of $\partial T$.
Let $h\in\G{\Gamma}=\Gamma$ be the map defined by:
$$
h(x)=\begin{cases}
g(x)&\text{if $x\in \partial T_{\vec e}$,}\\
g^{-1}(x)&\text{if $x\in g\partial T_{\vec e}$},\\
x&\text{otherwise.}
\end{cases}
$$
We see that $hg$ fixes $T_{\vec e}$, and thus the claim.
\end{proof}

\begin{remark}
Since, for all $q_1,q_2\geq 2$, the trees $T_{q_1}$ and $T_{q_2}$ are quasi-isometric, the
groups $\QI(T_{q_1})$ and $\QI(T_{q_2})$ are isomorphic.
\end{remark}

\section*{Acknowledgements}
The first author would like to thank
Mati Rubin for a fruitful discussion and
the Technion --- Israel Institute of Technology
for hospitality when working on the preliminary version of this paper.
The second author would like to thank
Hebrew University of Jerusalem for hospitality during the preparation of the paper.
The authors gratefully acknowledge the support from the Erwin Schr{\"o}dinger
Institute in Vienna at the final stage of the work, during the meeting
`Measured group theory 2016'.

\bibliography{bdd_simplicity}
\bibliographystyle{acm}

\end{document}